\documentclass{amsart}

\usepackage[latin1]{inputenc}
\usepackage{subfigure,amsmath,color}
\usepackage[comma,numbers]{natbib}
\usepackage{graphicx}
\usepackage{mathptmx}

\newtheorem{theorem}{Theorem} %[section]
\newtheorem{lemma}[theorem]{Lemma}

\theoremstyle{definition}\newtheorem{remark}[theorem]{Remark}
\theoremstyle{definition}\newtheorem{example}[theorem]{Example}

\newcommand{\bN}{\mathbb{N}}
\newcommand{\bP}{\mathbb{P}}

\newcommand{\bR}{\mathbb{R}}

\newcommand{\bZ}{\mathbb{Z}}
\newcommand{\bH}{\mathbb{H}}

\newcommand{\cA}{\mathcal{A}}

\newcommand{\cF}{\mathcal{F}}

\newcommand{\cL}{\mathcal{L}}
\newcommand{\cM}{\mathcal{M}}

\newcommand{\cT}{\mathcal{T}}

\newcommand{\Geo}{\text{\rm Geo}}
 
\newcommand{\todistr}{\to_{\text{\rm\tiny distr}}} 
\newcommand{\eqdistr}{=_{\text{\rm\tiny distr}}} 
 
\newcommand{\ct}{^{\text{\rm\tiny \hspace{-.25mm} co}}}
\newcommand{\df}{{\text{\rm d\hspace{.25mm}}}}

\newcommand{\brend}{\hfill $\triangleleft$} %\\[-2mm]}

\allowdisplaybreaks

\begin{document}

\title[]{Leader election:  A Markov chain approach}

\author{Rudolf Gr\"ubel}
\address{Institut f\"ur Mathematische Stochastik\\ 
         Leibniz Universit\"at Hannover\\
         Postfach 6009\\
         30060 Hannover\\
         Germany}

\email{rgrubel@stochastik.uni-hannover.de}

\author{Klaas Hagemann}
\address{Institut f\"ur Mathematische Stochastik\\ 
         Leibniz Universit\"at Hannover\\
         Postfach 6009\\
         30060 Hannover\\
         Germany}

\email{hage@stochastik.uni-hannover.de}

\subjclass[2000]{Primary 60J10, secondary 60J20, 60J50, 68W40}
\keywords{Boundary theory, election algorithms, geometric distribution, Markov chain, maxima, 
                  periodicity, tail $\sigma$-field.}
\date{\today}

\begin{abstract} 
A well-studied randomized election algorithm proceeds as follows: In each round the remaining
candidates each toss a coin and leave the competition if they obtain heads. Of interest is the
number of rounds required and the number of winners, both related to maxima of geometric
random samples, as well as the number of remaining participants as a function of the number of 
rounds. We introduce two related Markov chains and use ideas and methods from 
discrete potential theory to analyse the respective asymptotic behaviour as the initial number of
participants grows. One of the tools used is the approach
via the R\'enyi-Sukhatme representation of exponential order statistics,
which was first used in the leader election context by Bruss and Grübel in~\cite{BrGr03}.
\end{abstract}

\maketitle

\section{Introduction}\label{sec:intro}
We consider the following election algorithm: Starting with a group of size $n$, in each round the remaining
participants simultaneously toss a coin and leave the competition if it turns up heads, which we suppose 
to happen with probability $\theta\in (0,1)$ for the coins of all members of the group. In its simplest form the procedure 
ends if there is only one person left, in which case the winner is unique, or if all remaining participants 
obtain heads in the same round, when there would consequently be more than one winner.

This simple random election algorithm and its variants have received quite some attention over 
a period of more than a quarter of a century, see e.g.~\cite{BrOC90,ESS93,BES95,FMS96,KP96,BrGr03,Gne04,LP09,KM14,AKM15}. 
The duration of the original game is obviously related to the maximum of a sample of
geometric random variables, and the probability that the winner is unique is similarly related to the uniqueness
of this maximum. The earliest published paper on this topic that we are aware of is the paper~\cite{BrOC90} 
by Bruss and O'Cinneide, who refer to a presentation by P.\,N.\,Bajaj at an AMS meeting in 1988.  
They found that the probability that
there is a single winner, i.e.~that the maximum is unique, does not converge as $n\to\infty$
and that it is asymptotically logarithmically periodic, meaning in particular that the probability does 
converge along specific subsequences~$(n_k)_{k \in\bN}$. 

As most of the authors cited above, we are interested in the behaviour of the election algorithm as the number $n$ 
of participants grows to infinity, specifically in the use of Markov chain techniques. We now describe 
two chains that are relevant in this context.

First, we may think of an infinite number of participants who all toss their coins simultaneously, and we then 
regard the first $n$ of these. Specifically, with a sequence $(\xi_i)_{i\in\bN}$ of independent random variables, 
all geometrically distributed with parameter $\theta$, and with 
\begin{equation*}
  M_n \; :=\; \max\{\xi_1,\ldots,\xi_n\},
\end{equation*}
we obtain a representation of the number $R_n$ of rounds 
needed and the number $L_n$ of winners as 
\begin{equation}\label{eq:defML}
  L_n\;  :=\; \#\{1\le i \le n:\, \xi_i=M_n\}, \quad 
  R_n=\begin{cases} M_n, &\text{if } L_n=1,\\ M_n+1, &\text{if } L_n\ge 2,\end{cases} 
\end{equation}
if we start with $n$ participants (non-uniqueness of the maximum is noticed only after an additional round, with
the number of participants dropping from some $k>1$ to 0). 
The point here is that the values for different $n$'s are `coupled' in a manner that
leads to a Markov chain $(Y_n)_{n\in\bN}$, $Y_n:=(M_n,L_n)$, with state space $E=\bN\times\bN$ and transition
probabilities 
\begin{equation}\label{eq:transY}
  \bP \bigl( Y_{n+1}=(j,l)\big|Y_n=(i,k)\bigr) \, = \,
  \begin{cases}
    \theta(1-\theta)^{j-1} ,&\text{if } j>i, \, l=1,\\
      \theta(1-\theta)^{i-1},&\text{if } j=i, \, l=k+1,\\
      1- (1-\theta)^{i-1},&\text{if } j=i, \, l=k,\\
     0,&\text{otherwise.}
  \end{cases}
\end{equation}

Secondly, we consider the number $N_n$ of participants in round $n$ if we start with a group of size $k$, 
so that $N_1=k$. The process
$(N_n)_{n\in\bN}$ is again a Markov chain, now with state space $\{0,\ldots,n\}$ and transition  probabilities
\begin{equation}\label{eq:transN}
  \bP\bigl(N_{n+1}=j\big| N_n=i\bigr) \, = \, p(i,j) \, :=\, 
  \begin{cases} \displaystyle
    \binom{i}{j} \, \theta^{i-j}(1-\theta)^j, &\text{if } i>0,\, j=0,\ldots,i,\\
     \qquad 1, &\text{if } i=j=0,\\
     \qquad 0,&\text{otherwise.}
  \end{cases}
\end{equation}
It is easy to see that $N_n\to 0$ with probability~1 as $n\to\infty$, whatever the initial number $k$, 
so asymptotics  will refer to a sequence of such processes with $k\to\infty$ in this case.  

We will use discrete potential theory for the analysis of
the space-time behaviour of these chains. This area is also known as Markov chain boundary 
theory and has recently found many applications in the context of random combinatorial structures
that arise in the analysis of sequential algorithms; 
see~\cite{GrKMK} for a simple introduction. Our results give further examples for
the use of this approach. The chains in the present paper are more complicated than those
of the combinatorial type as we no longer have a locally finite transition tree, meaning that from
a given state there are infinitely many possible next states that can be visited with positive probability.    

The next section contains the necessary boundary theory background.
In Sections~\ref{sec:max} and Sections~\ref{sec:participants} respectively we 
then apply the theory to the two chains introduced above.  In terms of the election algorithm 
the first application relates to the existence of strong limit results for the duration and the number
of winners, the second contributes, in our view, to the understanding of the periodicity phenomenon
mentioned above. We will skip some technical details and refer the reader to~\cite{KH} for a full
treatment, together with various extensions.

\section{Boundary theory for space-time Markov chains}%
\label{sec:summ}

Doob's seminal paper~\cite{Doob59} may be regarded as the starting point of boundary theory for Markov 
processes with discrete time parameter and discrete state space. A recent and excellent textbook introduction 
to this circle of ideas is contained in~\cite{Woess09}, but see also the classic~\cite{KSK}. 
We give a brief outline of the main ideas, but omit details. 

The basic data consist of 
\begin{itemize}
\item a countable set $E$, the state space,
\item a probability measure on $E$, the initial distribution, represented by a sequence
          $q=(q(x))_{x\in E}$
          where $q(x)$ is the probability of the set~$\{x\}$, and
\item a matrix $P=(p(x,y))_{x,y\in E}$ of functions $p:E\times E\to \bR_+$, the transition matrix and 
transition probabilities, where the latter satisfy 
\begin{equation*}
  \sum_{y\in E} p(x,y) =1\quad\text{for all } x\in E.
\end{equation*}
\end{itemize}
Let $\Omega:= E^{\bN}$ be the set of all sequences $(x_n)_{n\in\bN}$ of elements of $E$. We endow $\Omega$
with the $\sigma$-field $\cA$ generated by the sets 
\begin{equation*}
  A(x_1,\ldots,x_k) := \bigl\{ (y_n)_{n\in\bN} \in \Omega:\, y_i=x_i\text{ for }i=1,\ldots,k\bigr\}
\end{equation*}
where $k\in\bN$ and $x_1,\ldots,x_k\in E$, and refer to the measurable space $(\Omega,\cA)$ as the path space. 
The basic data listed above provide a unique probability measure $\bP$ on the path space via
\begin{equation}\label{eq:cylprob}
  \bP\bigl(A(x_1,x_2,\ldots,x_k)\bigr) \,=\, q(x_1)\, \prod_{i=2}^k p(x_{i-1},x_i),
\end{equation}
for all $k\in\bN$, $x_1,\ldots,x_k\in E$. Using the projections
\begin{equation*}
  X_n:\Omega\to E,\ (x_m)_{m\in\bN} \mapsto x_n,
\end{equation*}
$n\in\bN$,  we then obtain a Markov chain $(X_n)_{n\in\bN}$ with state space $E$, initial distribution $q$ and transition matrix
$P$. We assume a weak form of irreducibility,
\begin{equation}\label{eq:weakirr}
  \bP(X_n=x\,\text{ for some }n\in\bN) > 0 \quad\text{for all } x\in E,
\end{equation}
which means that every state has a positive probability of being visited.

In this set-up we say that the chain is of \emph{space-time type} if the time parameter $n$ of the chain is a function
of its state $x\in E$. Equivalently we may consider the state space as being graded in the sense that $E$ is the disjoint
union of the segments $E_n$ of possible values of $X_n$. In particular, the transition mechanism is then adapted to the
grading in the sense that $p(x,y)>0$ implies that $x\in E_n$, $y\in E_{n+1}$ for some $n\in\bN$. It is well-known
(and trivial) that any Markov chain $(Y_n)_{n\in\bN}$, if augmented by the time parameter, gives a space-time chain
$(X_n)_{n\in\bN}$, i.e.\ we set $X_n:=(n,Y_n)$ for all $n\in\bN$.  Further, for a Markov chain $(Y_n)_{n\in\bN}$ that is not
homogeneous in time, meaning that the transition probabilities may depend on $n$, this augmentation provides time 
homogeneity. 

The above framework gives rise to three structures, with the corresponding problems of describing these in terms of familiar 
objects. 

For the first we recall that a function $h: E\to \bR$ is \emph{harmonic} if
\begin{equation*}
  h(x) = \sum_{y\in E} p(x,y) \, h(y)\quad\text{for all } x\in E.
\end{equation*}
We are interested in the set $\bH_{1,+}$ of such functions that are non-negative and normalized in the sense
that $\sum_{x\in E} q(x) h(x)=1$. This is a convex set, so describing $\bH_{1,+}$ could be its identification as a Choquet 
simplex together with a characterization of its extreme points. 

For the second object and problem we define the backwards transition matrix $P\ct$ and transition probabilities 
$p\ct (y,x)$, $x,y\in E$, by 
\begin{equation*}
  p\ct (y,x) := \bP(X_{n-1}=x|X_n=y), \quad P\ct=(p\ct (y,x))_{y,x\in E}.
\end{equation*}
These are also known as \emph{cotransition probabilities}.
The object of interest is now the set $\cM$ of all probability measures $\tilde\bP$ on the path space under which the
projections $(X_n)_{n\in\bN}$ become a Markov chain with the same backwards transitions as under the original $\bP$.
A straightforward computation shows that $\cM$ is again a convex set.

Finally, by a \emph{compactification} of the state space we mean a compact topological space 
$\bar E$ together with an injective mapping 
$\phi:E\to \bar E$ such that the image $\phi(E)$ is dense in $\bar E$ and that the trace of the $\bar E$ topology 
on this image
is equal to the discrete topology. The third structure we are interested in is a compactification of the state space in
which the $X_n$'s converge almost
surely as $n\to\infty$, and that is sufficiently detailed in the sense that the limit variable  $X_\infty$ generates
the tail $\sigma$-field of the process, meaning that
\begin{equation*}
          \sigma(X_\infty) =_{\text{\rm a.s.}} \cT(X) := \bigcap_{n=1}^\infty \sigma\bigl(\{X_m:\, m\ge n\}\bigr).
\end{equation*}

Interestingly, for space-time Markov chains these three questions are closely related.
First, for an $h\in\bH_{1,+}$ we may define a new transition mechanism, leading to the \emph{$h$-transform} 
of the original chain, by
\begin{equation}\label{eq:htrans}
  P_h=(p_h(x,y))_{x,y\in E}, \quad p_h(x,y) = \frac{1}{h(x)}\, p(x,y) \, h(y) .
\end{equation}
In order to not lose irreducibility in the sense of~\eqref{eq:weakirr}, we have to restrict the state space from $E$ 
to the set $E(h):=\{x\in E:\, h(x)>0\}$.
Let $\bP_h$ be the associated measure on the path space, so that under $\bP_h$ the process $(X_n)_{n\in\bN}$ 
is a Markov chain with transition matrix $P_h$. It is easy to see that~\eqref{eq:htrans} extends to $n$-step
transitions, which leads to 
\begin{equation}\label{eq:hdens}
  \frac{\df\bP_h^{(X_1,\ldots,X_k)}}{\df \bP^{(X_1,\ldots,X_k)}}(x) \, = \, h(x_k)
                   \quad\text{for all } x=(x_n)_{n\in\bN}\in \Omega.
\end{equation}
Another straightforward calculation shows that the transformed process
has the same backwards transition probabilities as $\bP$, so that $\bP_h\in\cM$. Conversely, for any $\tilde\bP\in\cM$
we obtain an element $h\in\bH_{1,+}$ via
\begin{equation}\label{eq:tildeP}
       h(x) := \frac{\tilde\bP(X_n=x)}{\bP(X_n=x)}\quad \text{for all } x\in E_n.
\end{equation}
Note that~\eqref{eq:weakirr} is important here, and that the definition of $h$ relies on the gradedness of the state space.
With the help of~\eqref{eq:hdens} it is easy to check that, apart from being bijective, the relationship 
is also  linear in the sense that
\begin{equation*}
  \bP_{\alpha h_1+(1-\alpha)h_2}\; =\; \alpha\, \bP_{h_1} + (1-\alpha) \,\bP_{h_2}\quad
               \text{for all } h_1,h_2\in\bH_{1,+}, \, 0<\alpha <1.
\end{equation*}
As a consequence the $h$-transform maps the extreme points of the convex sets $\bH_{1,+}$ and $\cM$ to 
each other. 
 
The starting point for the connection between harmonic functions and  the third problem is the observation that 
$(h(X_n),\cF_n)_{n\in\bN}$ 
is a non-negative martingale for all $h\in\bH_{1,+}$; here $\cF_n$ is the $\sigma$-field generated by the variables 
$X_1,\ldots,X_n$. Thus, by Doob's forward convergence theorem, $h(X_n)$ converges to some limit variable $Z$ almost
surely, where of course $Z$ may depend on $h$. 
Conversely, for an event $A\in\cT$ with $\kappa:=\bP(A)>0$ we obtain an element $h$ of $\bH_{1,+}$ 
via the following steps: By the Markov property, we may write the conditional expectation of the corresponding
indicator function as a function of $X_n$,  $\phi_n(X_n) =  E[1_A|\cF_n]$, and the space-time property implies
that we may then consistently define $h$ on the whole of $E$,
as in~\eqref{eq:tildeP}, by setting $h(x)=\kappa^{-1} \phi_n(x)$ for $x\in E_n$,
$n\in\bN$. 

The specific \emph{Doob-Martin compactification} now proceeds as follows:
We regard a sequence $(y_n)_{n\in\bN}\subset E$ as
convergent if, for all fixed $m\in\bN$ and $x_1,\ldots,x_m\in E$,
the conditional probabilities $\bP(X_1=x_1,\ldots,X_m=x_m|X_n=y_n)$ converge as $n\to\infty$. 
Due to the Markov property the construction can be based on the 
\emph{Martin kernel} $K$,
\begin{equation*}
  K(x,y) :=\frac{\bP(X_n=y| X_m=x)}{\bP(X_n=y)},
    \quad x,y \in E,\, n>m, 
\end{equation*}
where $m$ and $n$ are the time values associated with the states $x$ and $y$ respectively. Indeed, 
\begin{equation*}
  \bP(X_1=x_1,\ldots,X_m=x_m|X_n=y_n) = K(x_m,y_n)\, \bP(X_1=x_1,\ldots,X_m=x_m),
\end{equation*}
which connects the convergence condition for the conditional probabilities to the convergence of the values of the
Martin kernel, and which also exhibits the connection to the backwards transitions. The functions
$K(x,\cdot)$, $x\in E$, are bounded and separate the points of~$E$. The Stone-{\v C}ech procedure from 
general topology, or the introduction of a suitable metric on $E$ and subsequent completion, then lead to 
a compact space $\bar E$ that contains (a copy of) the discrete space $E$
and allows for a continuous extension of the functions $y\mapsto K(x,y)$, $x\in E$. 
Using the same symbol $K$ for the extended functions and lower case Greek letters for the boundary elements
we then obtain that all extremal elements $\bH_{1,+}$ are of the form $K(\cdot,\alpha)$ for some $\alpha$ in the
\emph{Martin boundary} $\partial E:=\bar E\setminus E$. In many cases the reverse implication also holds; 
in general one writes $\partial_{\text{\rm\tiny min}}E$ for the subset corresponding to the extremal elements; this is 
the \emph{minimal} boundary. With this notation in place, we can now state some central results: 
\begin{itemize}
\item[(R1)] For each 
$h\in \bH_{1,+}$ there exists a unique probability measure $\mu_h$ on (the Borel subsets of) $\bar E$ with 
$\mu_h(\partial_{\text{\rm\tiny min}}E)=1$ that represents $h$ in the sense of
\begin{equation*}
  h(x) = \int K(x,\alpha)\, \mu_h(d\alpha)\ \  \text{ for all } x\in E.
\end{equation*}

\item[(R2)] With respect to $\bP_h$ the variables $X_n$ converge almost surely to some $X_\infty$ with values in 
the boundary. 

\item[(R3)] The limit variable $X_\infty$ generates the tail $\sigma$-field up to $\bP_h$-null sets.

\item[(R4)] The distribution of $X_\infty$ is given by the measure $\mu_h$ that represents $h$.

\item[(R5)] The process conditioned on $X_\infty=\alpha\in\partial_{\text{\rm\tiny min}}E$ is an $h$-transform 
of the original chain, with the harmonic function given by $h=K(\cdot,\alpha)$.
\end{itemize}
The last of these may be rephrased as follows: The kernel
\begin{equation*}
  Q:\partial_{\text{\rm\tiny min}}E\times \cA\to [0,1],\quad (\alpha,A) \mapsto \bP_{K(\cdot,\alpha)}(A)
\end{equation*}
is a regular version of the conditional distribution of $X$ given $X_\infty$. This leads to an interpretation 
of the chain as a two-stage experiment, where we first select the final value and then run the corresponding transformed 
chain. There is an obvious analogy with classical and Bayesian statistics; see~\cite{Dynkin78} and~\cite{Lau88}.

\smallbreak

We end this sketch with a disclaimer: The beauty and elegance of the general theory notwithstanding, its actual
implementation for a specific Markov chain, such as a description of the boundary and the conditioned chains,
may be far from being trivial; see~\cite{EGW2} for a recent example. Independent of their applications in the context of
the election algorithm it may therefore be of interest that the two chains introduced in Section~\ref{sec:intro} 
permit a comparably short and explicit treatment along the above lines.

\section{The maximum and its multiplicities}%
\label{sec:max}

Recall the definition~\eqref{eq:defML} of our first Markov chain, with transition probabilities as given in~\eqref{eq:transY}.   
We write $\bar\bN:=\bN\cup\{\infty\}$ for the one-point compactification of $\bN$.

\begin{theorem}\label{thm:maxmult}
Let $X=(X_n)_{n\in\bN}$, with $X_n:=(n,M_n,L_n)$ for all $n\in\bN$, be the space-time chain associated with 
the process of maxima and their multiplicities.

\smallbreak
\emph{(a) } A sequence $(x_n)_{n\in\bN}$ of states $x_n=(n,j_n,l_n)\in E$ converges in the Doob-Martin topology
associated with $X$ if and only if, as $n\to\infty$,
\begin{itemize}
  \item $j_n$ converges to some $J\in \bar\bN$, and
  \item $l_n/n$ converges to some $\alpha\in [0,1]$ in the euclidean topology.
\end{itemize}

\smallbreak
\emph{(b) } The extended Martin kernel $K(\,\cdot\,; J,\alpha):E\to\bR_+$ associated with the limit point
$(J,\alpha)\in \bar\bN\times [0,1]$ is given by
\begin{align*}
  K(m,i,k;J,\alpha)\ &=\ (1-\alpha)^{m} \bigl(1-(1-\theta)^{J-1}\bigr)^{-m}, \quad\text{if } J\in\bN,\, J>i,\\
  K(m,J,k;J,\alpha) \ &= \ \alpha^k (1-\alpha)^{m-k} 
                       \bigl(\theta(1-\theta)^{J-1}\bigr)^{-k}\bigl(1-(1-\theta)^{J-1}\bigr)^{k-m},\quad\text{if } J\in\bN, \\
  K(m,i,k;\infty,\alpha)\ &=\ (1-\alpha)^{m},  
\end{align*}
and $K(m,i,k;J,\alpha)=0$ in all other cases.  

\smallbreak
\emph{(c) } As $n\to\infty$, $X_n$ converges almost surely to the fixed boundary point $(\infty,0)$.
 
\smallbreak
\emph{(d) } The extended Martin kernels associated with $(J,\alpha)$ are harmonic if $J\in\bN$, 
and the corresponding $h$-transform is the Markov chain with transition probabilities
\begin{equation}\label{eq:transh_Y}
  p^{(h)}(m,i,k;m+1,j,l)  \, = \,
  \begin{cases} \displaystyle
   \frac{(1-\alpha)(1- (1-\theta)^{i-1})}{1-(1-\theta)^{J-1}} ,&\text{if } i=j\le J, \; l=k,\\[3mm]
\displaystyle
   \frac{(1-\alpha)(1-\theta)^{i-1}\theta}{1-(1-\theta)^{J-1}} ,&\text{if } j=i <J, \; l=k+1,\\[3mm]
\displaystyle
   \frac{(1-\alpha)\theta(1-\theta)^{j-1}}{1-(1-\theta)^{J-1}} ,&\text{if } i<j <J, \; l=1,\\[3mm]
   \qquad\alpha,&\text{if } i < j,\; j=J, \; l=1,\\
   \qquad\alpha,&\text{if } i= j=J, \; l=k+1,\\
      \qquad 0,&\text{otherwise.}
  \end{cases}
\end{equation}
\end{theorem}

The Markov chain in Part (d) arises by conditioning on the limits $J$ for the maximum and $\alpha$ for its
relative multiplicity. 
A constructive interpretation can be obtained as follows: First, consider a sequence
$(\xi_n^J)_{n\in\bN}$ of independent random variables, which all have the geometric distribution with parameter
$\theta$, but now conditioned on the range $\{1,2,\ldots,J-1\}$. Let $Y^J=(Y_n^J)_{n\in\bN}$ be the bivariate
process of maxima and their multiplicities based on  $(\xi_n^J)_{n\in\bN}$. Consider a second chain 
$Z^J=(Z_n^J)_{n\in\bN}$ that moves from $(i,k)$, $i<J$, to $(J,1)$ with probability $1$, and then moves 
upwards on the vertical line $\{J\}\times \bN$ with probability $\alpha$ resp.\ stays where it is with 
probability $1-\alpha$. As long as we are inside the strip $\{1,\ldots,J-1\}\times\bN$ we follow $Y^J$ with 
probability $1-\alpha$ but switch to $Z^J$ with probability $\alpha$. Note the double role of the parameter~$\alpha$.

The extended Martin kernel associated with a pair $(\infty,\alpha)$ is not harmonic but only superharmonic,
which means that the corresponding transition matrix is strictly substochastic. Nevertheless we can arrive at an
interpretation along the lines just given for $(J,\alpha)$ with $J\in\bN$ if we augment the state space by
admitting the value $\infty$ for the maximum: The transformed process would be based on the same
$\xi$-variables as the original process, but we would jump to $(\infty,1)$ with probability $\alpha$ and
then stay or move upwards on $\{\infty\}\times \bN$ as in the case $J\in\bN$.

For the proof of the theorem  we need the following elementary lemma.

\begin{lemma}\label{lem:bin} 
Let $m\in\bN$ be fixed and consider a sequence $(l_n)_{n\in\bN}\subset \bN$
with $l_n\le n-m$ for all $n\in\bN$. Then $\binom{n-m}{l_n}\binom{n}{l_n}^{-1}$ converges if and only if 
$l_n/n$ converges, both as $n\to\infty$. Further, if $\,\lim_{n\to\infty}l_n/n=\alpha$, then
\begin{equation}\label{eq:lemma1}
  \lim_{n\to\infty} \binom{n-m}{l_n}\binom{n}{l_n}^{-1} =\, (1-\alpha)^m,
\end{equation}
and, for all $k\in\bN$,
\begin{equation}\label{eq:lemma2}
  \lim_{n\to\infty} \binom{n-m}{l_n-k}\binom{n}{l_n}^{-1} =\, \alpha^k(1-\alpha)^m.
\end{equation}
\end{lemma}

\begin{proof} We have
  \begin{equation}\label{eq:lemma}
    \binom{n-m}{l_n}\binom{n}{l_n}^{-1}\, = \, \prod_{r=0}^{l_n-1} \frac{n-m-r}{n-r}
      \, =\, \exp\biggl(\,\sum_{r=0}^{l_n-1} \log\biggl( 1 - \frac{m}{n - r}\biggr)\biggr).
  \end{equation}
Suppose that $l_n/n\to\alpha\in(0,1)$ and let
\begin{equation*}
  R(n,m,r) \, :=\,  \log\biggl( 1 - \frac{m}{n-r}\biggr) \; +\;  \frac{m}{n-r}.
\end{equation*}
We will use the inequality $\bigl| \log(1-x) + x \bigr| \le x^2$ which is valid for $0\le x\le 1/2$.
In view of $\alpha<1$ we can find $n_0=n_0(m)$ such that $\; 0 \le m/(n-r) \le 1/2\; $ for all $n\ge n_0$,
and for such~$n$ 
\begin{equation*}
  \sum_{r=0}^{l_n-1} \,\bigl| R(n,m,r)\bigr| \; \le \; \frac{m^2l_n}{(n-l_n)^2},
\end{equation*}
which is $o(1)$ in view of $\alpha<1$. Further,
\begin{equation*}
  \sum_{r=0}^{l_n-1} \frac{1}{n-r} \ 
     = \ \int_0^{l_n/n} \frac{1}{1-\frac{\lfloor nx\rfloor}{n}} \, dx\
     \to\ \int_0^\alpha \frac{1}{1-x}\, dx \ = \ -\log(1-\alpha),                            
\end{equation*}
so that, taken together,
\begin{equation*}
  \lim_{n\to\infty} \sum_{r=0}^{l_n-1} \log\biggl( 1 - \frac{m}{n - r}\biggr)\; =\; -m \,\log(1-\alpha).
\end{equation*}
Hence the limit of the ratio of the binomial coefficients exists and is equal to $(1-\alpha)^m$ if 
$l_n/n\to\alpha\in(0,1)$. Monotonicity consideration can be used to extend this to the boundaries 
$\alpha=0$ and $\alpha=1$. Taken together this proves~\eqref{eq:lemma1}; the modifications needed 
for~\eqref{eq:lemma2} should be obvious. 

For the proof of the convergence condition we note that 
\begin{equation*}
\alpha_-:=\liminf_{n\to\infty}\,\frac{l_n}{n}\; <\; \limsup_{n\to\infty}\frac{l_n}{n}=:\alpha_+  
\end{equation*}
would imply the existence of two subsequence such that the corresponding ratios of binomial
coefficients in~\eqref{eq:lemma1} converge to $(1-\alpha_-)^m$ and $(1-\alpha_+)^m$ respectively. 
\end{proof}

\begin{proof}[Proof of Theorem~\ref{thm:maxmult}]
In order to have maximal value $j$ with multiplicity $l$ at time $n$ we need exactly $l$ of the variables 
$\xi_1,\ldots,\xi_n$ to be equal to $j$ and all others to be strictly less than $j$. This gives
\begin{equation*}
  \bP(M_n=j,L_n=l) \, =\, \binom{n}{l} \bigl(\theta(1-\theta)^{j-1}\bigr)^l\bigl(1-(1-\theta)^{j-1}\bigr)^{n-l}.
\end{equation*}
Similarly, an advance from $(M_m,L_m)=(i,k)$ to $(M_n,L_n)=(i,l)$ with $l\ge k$ has probability
\begin{align*}
  \bP(M_n=i,L_n=l|&M_m=i,L_m=k) \\ 
           &=\  \binom{n-m}{l-k} \bigl(\theta(1-\theta)^{i-1}\bigr)^{l-k}\bigl(1-(1-\theta)^{i-1}\bigr)^{n-m-(l-k)},
\end{align*}
so that
\begin{equation}\label{eq:K1finite}
  K(m,i,k;n,i,l) \ = \ \binom{n-m}{l-k} \binom{n}{l}^{-1}  
                          \bigl(\theta(1-\theta)^{i-1}\bigr)^{-k}\bigl(1-(1-\theta)^{i-1}\bigr)^{k-m}.
\end{equation}
If $l=l_n$ depends on $n$ such that $l_n/n\to \alpha\in [0,1]$ then, by Lemma~\ref{lem:bin}, 
for any fixed $m,i,k$ this converges as $n\to\infty$ to 
\begin{equation*}
  K(m,i,k;\infty,i,\alpha) \ = \ \alpha^k (1-\alpha)^{m-k} \bigl(\theta(1-\theta)^{i-1}\bigr)^{-k}\bigl(1-(1-\theta)^{i-1}\bigr)^{k-m}.
\end{equation*}
Similarly, if the value of the maximum increases, so that $j>i$, then for $l=1,\ldots,n-m$, 
\begin{align*}
  \bP(M_n=j,L_n=l|&M_m=i,L_m=k) \\ 
           &=\  \binom{n-m}{l} \bigl(\theta(1-\theta)^{j-1}\bigr)^{l}\bigl(1-(1-\theta)^{j-1}\bigr)^{n-m-l},
\end{align*}
and we arrive at
\begin{equation}\label{eq:K2finite}
  K(m,i,k;n,j,l) \ = \ \binom{n-m}{l} \binom{n}{l}^{-1}  
                          \bigl(1-(1-\theta)^{j-1}\bigr)^{-m}.
\end{equation}
By Lemma~\ref{lem:bin} again,  if $l_n/n\to \alpha\in [0,1]$, then this converges for any fixed $m,i,k$ 
if $j>i$ and the limit is
\begin{equation*}
  K(m,i,k;\infty,j,\alpha) \ = \  (1-\alpha)^{m} \bigl(1-(1-\theta)^{j-1}\bigr)^{-m}.
\end{equation*}
In the other direction this lemma, together with~\eqref{eq:K1finite} and~\eqref{eq:K2finite}, also shows that in 
order for $K(m,i,k;n,J,l_n)$ to converge for $i=1,\ldots,J$ it is necessary that $l_n/n$ converges to some value 
in $[0,1]$.

Now suppose that $(j_n)_{n\in\bN}$ is such that $j_n\to\infty$ as $n\to\infty$. Then 
\begin{equation*}
  \lim_{n\to\infty} \bigl(1-(1-\theta)^{j_n}\bigr)^{-m} = 1\quad\text{for all } m\in\bN. 
\end{equation*}
This together with~\eqref{eq:K2finite} shows that $K(m,i,k;n,j_n,l_n)$ converges for all $m,i,k$ if and only if 
$\binom{n-m}{l_n}\binom{n}{l_n}^{-1}$ converges, which implies the $J=\infty$ part of the first assertion.

For (a) and (b) it remains to show that for sequences $(j_n)_{n\in\bN}$ that do not converge in the one-point compactification 
of $\bN$ convergence of $K(m,i,k;n,j_n,l_n)$ cannot hold for all $m,i,k$, regardless of
the behaviour of $(l_n)_{n\in\bN}$. 

If  $(j_n)_{n\in\bN}$ does not converge in $\bar\bN$ then there are sequences 
$(n_s)_{s\in\bN}, (n_t)_{t\in\bN}\subset\bN$ with $n_s,n_t\to\infty$ such that
\begin{equation}\label{eq:Jdiv}
  J_- := \lim_{s\to\infty} j_{n_s} = \liminf_{n\to\infty} j_n\  <\  \limsup_{n\to\infty} j_n =\lim_{t\to\infty} j_{n_t}=:J_+. 
\end{equation}
In particular, $J_-\in\bN$ and then necessarily, for some $s_0\in\bN$, $j_{n_s}=J_-$ for all $s\ge s_0$.
Convergence of $(n,j_n,l_n)$ in the Doob-Martin topology means that 
$K(m,i,k;n,j_n,l_n)$ converges for \emph{all} $m$, $i$ and $k$. For $m=1$ we obtain
\begin{equation}\label{eq:Kexpr1}
  K(1,i,k;n,j,l) \; =\; \frac{1-l/n}{1-(1-\theta)^j} \quad \text{ if } i<j,
\end{equation}
and $K(1,i,k;n,j,l)=0$ if $i>j$. In view of~\eqref{eq:Jdiv} both alternatives happen infinitely often 
for the sequence $(j_n)_{n\in\bN}$ if we choose $i$ such that $J_-<i<J_+$, 
so that only the value 0 is possible for the limit. Since the denominator in~\eqref{eq:Kexpr1} is bounded by 1 and 
also bounded away from 0,
we must therefore have $\lim_{n\to\infty} l_n/n=1$. However, this in turn implies
\begin{equation*}
  \lim_{s\to\infty} K(1,J_-,1;n_s,,j_{n_s},l_{n_s}) \; = \; \lim_{s\to\infty} K(1,J_-,1;n_s,J_-,l_{n_s}) 
          \; = \; \theta^{-1}\, (1-\theta)^{J_--1} \ > \ 0,
\end{equation*}
which is a contradiction. 

The remaining case where  $J_+=J_-+1$ can be handled similarly.
 
The above proves parts (a) and (b). For the proof of (c) we first note that the general theory implies 
that $X_n=(M_n,L_n)\to X_\infty$ almost surely for some $X_\infty$ with values in the boundary. For $M_n$, it is therefore 
enough to prove convergence in distribution, which is straightforward from $\bP(M_n>k)\to 1$ for all fixed 
$k\in\bN$. Clearly, $L_n=1$ for infinitely many $n$ on the set $M_n\uparrow \infty$, so the almost sure
limit for $L_n/n$ can only have the value 0.  

Finally, (d) follows by calculation.
\end{proof}

Boundary theory for randomly growing discrete structures can sometimes be used to amplify a known
result on distributional convergence to an almost sure statement; see e.g.~\cite{GrMTree} in connection 
with the Wiener index of search trees. Obviously such a strengthening to a pathwise result is not possible
if the tail $\sigma$-field of the sequence of interest is trivial in the sense that it consists of events with 
probability~0 or~1 only. The above result therefore shows that the known distributional limit theorems
for the maximum and its multiplicity (along suitable subsequences) do not hold with probability one.
Indeed, there are no deterministic transformations of $M_n$ or $L_n$, even if we allow additional dependence
on $n$ or pass to a subsequence, that lead to a non-degenerate strong limit.

\begin{remark}\label{rem:asconv}
(a) Boundary theory can lead to an explicit description of the tail $\sigma$-field of a sequence of random variables,
if these constitute a Markov chain. However, if interest is mainly in the qualitative aspect of tail triviality then there 
are general results that can be used. Indeed, in the case at hand it is easy to see that
the terminal $\sigma$-field associated with the sequence $(M_n,L_n)_{n\in\bN}$ is contained 
in the exchangeable $\sigma$-field associated with the sequence $(\xi_i)_{i\in\bN}$, and,  
by the Hewitt-Savage theorem, see e.g.\ \cite[Section 3.9]{Breiman68},
i.i.d.\ sequences have a trivial exchangeable $\sigma$-field.  

(b) It is known, see e.g.~\cite[Corollary 7.51]{Woess09}, that a boundary point is an element of 
$\partial_{\text{\rm\tiny min}}E$ if and only if the tail $\sigma$-field of the corresponding $h$-transform
is trivial. The explicit construction for $h=K(\cdot,\cdot;\infty,J,\alpha)$ given above makes it possible
to use the Hewitt-Savage theorem to show that all these boundary points are in fact minimal.
\brend
\end{remark}

\section{The number of participants}%
\label{sec:participants}

The election algorithm motivates the following question: If we start at time $n=1$ with a group of size $j$,
what is the number of remaining participants, considered as a function of the number of rounds carried out? 
For example, the total duration of the 
procedure may then be written as the entrance time of the process into the set $\{0,1\}$, see~\eqref{eq:defML}.  

Again, we are interested in asymptotics, meaning here that $j=j_k$ is large. This situation differs from the one
in Section~\ref{sec:max} where we considered the asymptotic behaviour of a fixed chain. Two possibilities 
offer themselves: We can `invert time' and aim for a description of the set $\cM$ defined in Section~\ref{sec:summ},
or we shift the first time index from $n=1$ to some negative value depending on $j$, with the hope that
a limit object emerges for the resulting sequence of processes. We will do both.

For the first approach we start with the state space $E:=\bN\times\bN_0$ and consider Markov chains $X$ with 
backwards transition probabilities 
\begin{equation}\label{eq:bw}
  \bP\bigl(X_n=(n,j)\big| X_{n+1}=(n+1,i)\bigr)\; =\; \binom{i}{j} \, (1-\theta)^j\, \theta^{i-j}
 \end{equation}
for all $n\in\bN$, $ i\in\bN_0$ and  $j\in \{0,\ldots,i\}$. To simplify the notation we introduce
\begin{equation*}
  c(\theta)\, :=\, -\frac{1}{\log(1-\theta)}.
\end{equation*}
Note that $c(\theta)\log y$ is equal to the logarithm of $y>0$ with respect to the base $1/(1-\theta)$.
We recall that the associated state space augmentation is a compact topological space that is unique only  
up to homeomorphisms. In the representation given below we write $\bR^\star$ for the one-point (!) 
compactification of the real line. Formally, we have some object $\diamond$ that is not an element of $\bR$, 
we set $\bR^\star=\bR\sqcup \{\diamond\}$, we retain the euclidean topology on $\bR$, and, finally, we regard 
any sequence $(x_n)_{n\in\bN}$ with
\begin{equation*}
  \#\bigl\{n\in\bN:\, x_n\in [a,b]\bigr\} <\infty \quad\text{for all } a,b\in\bR
\end{equation*}
as convergent with limit $\diamond$. We will also 
need the probability densities $f_l$, $l\in\bN$, given by
\begin{equation}\label{eq:dens}
  f_l(x) \, =\, \frac{1}{(l-1)!}\, \exp\bigl(-lx-e^{-x}\bigr), \quad x\in\bR,
\end{equation}
and Euler's constant $\gamma:=\lim_{n\to\infty} (H_n-\log n)$, where $H_n:=\sum_{k=1}^n 1/k$ denotes
the $n$th harmonic number. 

\begin{theorem}\label{thm:part}  Let $X$ be a Markov chain with backwards transition probabilities 
given by~\eqref{eq:bw}.

\smallbreak
\emph{(a) } The Doob-Martin boundary $\partial E$ of $E$ associated with the space-time version of $X$ 
is $\bR^\star$, where a sequence of states
$(n_k,j_k)$, $k\in\bN$, converges to $z\in \bR^\star$ if and only if
\begin{equation}\label{eq:DMpart}
  \lim_{k\to\infty} \bigl(c(\theta)\log(j_k)\, -\, n_k\bigr) \; = \; z.
\end{equation}

\smallbreak
\emph{(b) } Let $W_i$ and $\zeta_i$ be independent random variables, where $\zeta_i$ is exponentially distributed 
with parameter $i$ and $W_i$ has density $f_i$, see~\eqref{eq:dens}. Then, for $z\in\bR$, the extended 
Martin kernel can be written as
\begin{equation}\label{eq:formK}
  K(m,i;z) \; = \; \frac{1}{\bP(X_m=i)} \, \bP\bigl(W_i< c_\infty(m,i;z) < W_i+\zeta_i\bigr)
\end{equation}
for all $m\in\bN$, $i\in\bN$, and with 
\begin{equation}\label{eq:cinfty}
  c_\infty(m,i;z) \,:=\,  H_i-\gamma - \frac{m+z}{c(\theta)}.
\end{equation}
Further, $K(m,i;\diamond)\equiv 0$.

\smallbreak
\emph{(c) } The extended Martin kernels $K(\cdot,\cdot; z)$ are harmonic if $z\in\bR$. 
\end{theorem}

\begin{proof} We have
\begin{equation*}
   K(m,i;n,j) \; = \; \frac{\bP(X_n=j|X_m=i)}{\bP(X_n=j)}\; = \; \frac{\bP(X_m=i|X_n=j)}{\bP(X_m=i)},     
\end{equation*}
so we need the $n$-step transition probabilities associated with
\begin{equation*}
  p(i,j) \, = \,
  \begin{cases} \displaystyle
    \binom{i}{j} \, \theta^{i-j}(1-\theta)^j, &i>0,\, j=0,\ldots,i,\\
     \qquad 1, &i=j=0,\\
     \qquad 0,&\text{otherwise.}
  \end{cases}
\end{equation*}
We consider first the case $j_k\to\infty$. Again, let $\xi_l$, $l\in\bN$, be independent random 
variables, all geometrically distributed with parameter $\theta$. Then, with
\begin{equation}\label{eq:kappa}
  \kappa(m,i;n_k,j_k) \, :=\, \bP(X_m=i|X_{n_k}=j_k) \, =\,  \bP(X_m=i) \, K(m,i;n_k,j_k) 
\end{equation}
we have $\kappa(m,i;n_k,j_k) = \bP\bigl(\#\{1\le l\le j_k:\, \xi_l > n_k-m\}=i\bigr)$, for all $j_k\ge i$, $n_k\ge m$.
As in~\cite{BrGr03} we now use the well-known relation between geometric and exponential random variables
together with the R\'enyi-Sukhatme representation of the order statistics of a sample from an exponential distribution;
see e.g.~\cite[p.336]{ShoWell} for the latter.
For this, we start with a sequence $(\eta_i)_{i\in\bN}$ of independent random variables, all exponentially distributed
with parameter~1. Then $(\lceil c(\theta)\eta_i\rceil)_{i \in\bN}$ is equal 
in distribution to $(\xi_i)_{i\in\bN}$, and
\begin{align*}
  \#\{1\le l\le j_k:\, \xi_l\ge n_k-m\}\ &=\ \#\{1\le l\le j_k:\, \lceil c(\theta)\eta_l\rceil> n_k-m\}\\
               &=\ \#\{1\le l\le j_k:\, \eta_l\ge c(\theta)^{-1} (n_k-m)\},
\end{align*}
Writing $\eta_{(j_k:1)}< \eta_{(j_k:2)} < \cdots <\eta_{(j_k:j_k)}$
for the increasing order statistics associated with $\eta_1,\ldots,\eta_{j_k}$ we therefore have that
\begin{equation}\label{eq:kappaord}
  \kappa(m,i;n_k,j_k) = \bP\bigl(\eta_{(j_k:j_k-i)} \le c(\theta)^{-1} (n_k-m) < \eta_{(j_k:j_k-i+1)} \bigr).
\end{equation}
The R\'enyi-Sukhatme representation says that the random vector 
$(\eta_{(j_k:1)},\eta_{(j_k:2)},\ldots,\eta_{(j_k:j_k)})$ is equal in distribution to the random vector
\begin{equation*}
  (V_{j_k},V_{j_k}+V_{j_k-1}, \ldots, V_{j_k}+V_{j_k-1}+\cdots +V_1),
\end{equation*}
with $V_1,\ldots,V_{j_k}$ independent, and where $V_l$ is exponentially distributed with parameter $l$,
$l=1,\ldots,j_k$. In particular,
\begin{equation*}
  \eta_{(j_k:j_k-i)} \eqdistr V_{i+1}+V_{i+2}+\cdots +V_{j_k}.
\end{equation*}
It is easy to see that $(M_{j_k,i})_{j_k> i}$ with
\begin{equation*}
  M_{j_k,i} \, := \, \sum_{l=i+1}^{j_k} \Bigl(V_l-\frac{1}{l}\Bigr), \quad j_k > i,
\end{equation*}
is a martingale that is bounded in $L^2$. Hence $M_{j_k,i}$ converges almost surely as $k\to\infty$ 
to some square integrable random variable $W_i$. The representation further implies that
$\eta_{(j_k:j_k-i+1)}-\eta_{(j_k:j_k-i)}$ is independent of $\eta_{(j_k:j_k-i)}$ and that it has an exponential
distribution with parameter $i$. Taken together this gives the following convergence in
distribution,
\begin{equation*}
\Bigl(\eta_{(j_k:j_{k-i})}-\sum_{l=i+1}^{j_k}\frac{1}{l}\, ,\, \eta_{(j_k:j_{k-i+1})}-\sum_{l=i+1}^{j_k}\frac{1}{l}\Bigr)
      \; \todistr\; (W_i,W_i+\zeta_i)\quad \text{as } k\to\infty, 
\end{equation*}
with $W_i$ and $\zeta_i$ independent, and $\zeta_i$ exponentially distributed with parameter $i$.
This implies the convergence of the probabilities in~\eqref{eq:kappaord} if
\begin{equation*}
       c_k(m,i) \; := \; \frac{n_k-m}{c(\theta)} \, -\, \sum_{l=i+1}^{j_k}\frac{1}{l}
\end{equation*}
converges (a detailed argument would use Slutsky's Lemma together with the continuity of the distribution 
functions of $W_i$ and $W_i+\zeta_i$). 
Recalling the expansion $H_j =  \log j + \gamma + o(1)$
as $j\to\infty$, we see that this is equivalent to the convergence of $c(\theta)\log j_k -n_k$ 
to some $z\in\bR$, and that we then have
\begin{equation*}
     \lim_{k\to\infty} c_k(m,i) =  c_\infty(m,i;z)
\end{equation*}
with $c_\infty(m,i;z)$ as in~\eqref{eq:cinfty}, and consequently
\begin{equation*}
  \lim_{k\to\infty} \kappa(m,i;n_k,j_k)\; =\; \bP\bigl(W_i < c_\infty(m,i;z) < W_i+\eta_i\bigr).
\end{equation*}
The results in~\cite[p.1258]{BrGr03} imply that $W_i$ has density $f_i$. Taken together this shows that the
convergence in (a) implies the convergence of the Martin kernels, together with the formula given
in~\eqref{eq:formK}. 

It remains to show  that the convergence condition in (a) is also necessary.  
As in the proof of Theorem~\ref{thm:maxmult} we recall that Doob-Martin
convergence implies the convergence of the Martin kernels for \emph{all} $m$ and $i$. Below we will use
$i=1$ and choose $m$ large enough. 

Suppose that we have
\begin{equation*}
 -\infty \, < \,  z_-:=\liminf_{k\to\infty} \bigl(c(\theta)\log(j_k) - n_k\bigr) 
      \; < \; \limsup_{k\to\infty} \bigl(c(\theta)\log(j_k) - n_k\bigr) =: z_+ \, <\, +\infty.
\end{equation*}
In particular, with suitably chosen subsequences $(k_+(l))_{l\in\bN}$ and $(k_-(l))_{l\in\bN}$
we would obtain $z_+$ and $z_-$ as limits of $c(\theta)\log(j_k) - n_k$ along $k=k_+(l)$
and $k=k_-(l)$ as $l\to\infty$. From $z_-<z_+$ it follows that $c_\infty(m,1;z_+) < c_\infty(m,1;z_-)$ for all $m\in\bN$.
We now note that the density of $W_1$ is strictly positive and unimodal, with argmax at $0$. 
This implies that the function
\begin{equation*}
  y\; \mapsto \; \bP(W_1< y < W_1+\zeta_1) \, 
                               = \, \int_0^\infty \bigl(\bP(W_1\le y) -\bP(W_1\le y-s)\bigr) \, e^{-s} \, ds
\end{equation*}
is strictly increasing on an interval $(-\infty,0)$. 
With $m$ chosen large enough, both $c_\infty(m,1;z_+)$ and $c_\infty(m,1;z_-)$ are less than 0,
so that
\begin{equation*}
  \bP\bigl(W_1 < c_\infty(m,1;z_+) < W_1+\eta_1\bigr)
           \; <\;  \bP\bigl(W_1 < c_\infty(m,1;z_-) < W_1+\eta_1\bigr).
\end{equation*}
Putting things together we see that $K(m,1;n_{k_+(l)},j_{k_+(l)})$ and $K(m,1;n_{k_-(l)},j_{k_-(l)})$ converge 
to different values. Hence $(n_k,j_k)$ does not converge in the Doob-Martin topology.

Similar arguments work in the other cases, where $z_-$ and $z_+$ may not be finite. 

Further, the argument 
in the necessity proof can also be used to show that different $z$-values lead to different extended kernels,
and it is straightforward to show that pointwise convergence of a sequence $(K(\cdot,z_n))_{n\in\bN}$ of 
extended kernels is equivalent to the convergence of $(z_n)_{n\in\bN}$ in $\bR^\star$.

Part (c) follows by calculation, see~\cite{KH}.
\end{proof}

In contrast to the situation in Section~\ref{sec:max} we did not start with a specific chain $X$ or probability
measure $\bP$ on $(\Omega,\cA)$, so it is not clear whether a chain with the postulated 
cotransitions exists at all.

\begin{example}\label{ex:forward}
Let us write $\Geo_0(\eta)$ for the number of failures version of the geometric distribution with parameter,
so that $V\sim \Geo_0(\eta)$ means 
\begin{equation*}
  \bP(V=i) = (1-\eta)^i \eta\qquad\text{for all } i\in\bN_0.
\end{equation*}
Let $W$ be another random variable, defined on the same probability space as $V$, with
\begin{equation*}
  \bP(W=j|V=i) = \binom{i}{j} (1-\theta)^j\theta^{i-j} \quad \text{for all } i\in\bN_0,\; j\in\{0,\ldots,i\}.
\end{equation*}
Taken together, the distribution of $V$ and the conditional distribution of $W$ given $V$ determine the
distribution of $W$. After some straightforward calculations we obtain
\begin{equation*}
  \bP(W=j) \; =\; \sum_{i=j}^\infty \binom{i}{j} (1-\theta)^i\theta^{i-j} (1-\eta)^i \eta
                  \; = \; \biggl(1-\frac{\eta}{1-\theta+\eta\theta}\biggr)^j \frac{\eta}{1-\theta+\eta\theta}, 
\end{equation*}
which shows that $W\sim\Geo_0(\zeta)$ with $\zeta:=\eta/(1-\theta+\eta\theta)$. 

The inverse $\psi_\theta$ of the function $\eta\mapsto \zeta$ is given by
\begin{equation*}
  \psi_\theta(\zeta) = \frac{\zeta(1-\theta)}{1-\zeta\theta}.
\end{equation*}
It is easy to check that $\psi_\theta$ is continuous and strictly increasing, with $\psi_\theta(0)=0$,
$\psi_\theta(1)=1$. 
For a given $\zeta_1\in (0,1)$ we define the sequence $(\zeta_n)_{n\in\bN}$ recursively
by $\zeta_{n+1}=\psi_\theta(\zeta_n)$ for all $n\in\bN$. The above calculation, together with 
Kolmogorov's consistency theorem, now provides the existence of a probability measure $\bP_\zeta$
on the path space such that, under this measure, $(X_n)_{n\in\bN}$ is a Markov chain with the
required backwards transitions; moreover, by construction $X_n\sim\Geo_0(\zeta_n)$ for all $n\in\bN$. 
This chain is not homogeneous in time, but our basic object is the corresponding space-time
chain $((n,X_n))_{n\in\bN}$, which is automatically homogeneous in time; see~\cite{KH} for more
details.
\brend  
\end{example}

For the second approach we consider a sequence of Markov chains with transition probabilities as 
in~\eqref{eq:transN} and with start in state $j_k$ at time $-k+1$; here we regard the sequence $(j_k)_{k\in\bN}$
as given. For a formal treatment we replace the time range $\bN$ that we have used so far by
the full set $\bZ$ of all integers. The path space is now $\Omega=\bN_0^\bZ$, the projections give a two-sided 
sequence $(X_n)_{n\in\bZ}$, these generate the $\sigma$-field $\cA$ on $\Omega$, and probability measures 
on $(\Omega,\cA)$ are determined by the values they assign to sets
of the form
\begin{equation}\label{eq:cyl2}
  A(i_l,\ldots,i_m) =\bigl\{(j_n)_{n\in\bZ}\in\Omega:\, i_s=j_s\text{ for } s=l,\ldots,m\bigr\},
\end{equation}
with $l,m\in\bZ$, $l<m$, and $i_l,\ldots,i_m\in\bN_0$. In particular, for each $k\in\bN$ we can define a probability 
measure $\bP_k$ on $(\Omega,\cA)$ via
\begin{equation*}
  \bP_k\bigl(A(i_l,\ldots,i_m)\bigr) \; =\; \prod_{s=-k+1}^{m-1} p(i_s,i_{s+1}),
\end{equation*}
with $p$ as in~\eqref{eq:transN}, whenever $l\le -k+1$ and $i_l=\cdots=i_{-k+1}=j_k$; to sets with $i_l\not=j_k$ for some
$l\le -k+1$ we assign the value 0. Under $\bP_k$, the coordinate
process $(X_n)_{n\in\bZ}$ models the sequence of participant numbers if we have a fixed number $j_k$ up to time $-k+1$ 
and then start the coin-tossing selection procedure.  Clearly, the paths are decreasing and the state 0 is absorbing.
Because of the extension to time values before $-k+1$ these measures
are all defined on the same measurable space. If we endow $\bN_0$ with the discrete topology and $\Omega$ with
the corresponding product topology then we obtain a topological space so that weak convergence of probability measures
is defined. It is easy to see that a sequence of probability measures on $(\Omega,\cA)$ 
converges weakly with respect to this topology if and only if the probabilities of all sets of the type given in~\eqref{eq:cyl2} converge. 

The following result is essentially a reformulation of parts of Theorem~\ref{thm:part} from this alternative point of view.

\begin{theorem}\label{thm:minusinf}
For a given sequence $(j_k)_{k\in\bN}$ of positive integers let $(\bP_k)_{k\in\bN}$ be the associated sequence of probability
measures on $\bN_0^\bZ$, where $\bP_k$ is such that $\bP_k(X_n=j_k)=1$ for all $n\le -k+1$, and
$(X_n)_{n\ge -k+1}$ is a Markov chain with transition probabilities as in~\eqref{eq:transN}.

Then $(\bP_k)_{k\in\bN}$ converges weakly as $k\to\infty$ if and only if, for some $z\in\bR$,
\begin{equation*}
  \lim_{k\to\infty} \bigl(c(\theta)\log j_k -k\bigr) \, =\, z.
\end{equation*}
Under the limit measure $\bP^z$ the full sequence $(X_n)_{n\in\bZ}$ is a Markov chain with transitions as
in~\eqref{eq:transN}, and, almost surely with respect to $\bP^z$,
\begin{equation*}
  \lim_{n\to -\infty} \bigl(c(\theta)\log X_{n}+n\bigr) \, =\, z.
\end{equation*}
\end{theorem}

\begin{proof}
Let $A=A(i_l,\ldots,i_m)$ be as in~\eqref{eq:cyl2}. Then, for all $k\in\bN$ with $-k<l$, and with $\kappa$ 
as in~\eqref{eq:kappa},
\begin{equation*}
  \bP_k(A) = \kappa(-l,i_l;k,j_k).
\end{equation*}
In particular, weak convergence as $k\to\infty$ of $\bP_k$ is equivalent to convergence of
$\kappa(-l,i_l;k,j_k)$ for all $l\in-\bN$, $i_l\in\bN_0$, so that the first statement is
an immediate consequence of part (a) of Theorem~\ref{thm:part}.

Clearly, the Markov property is not lost and neither does  the transition mechanism change when passing to the limit $\bP^z$
of a converging sequence $(\bP_k)_{k\in\bN}$. This gives the second statement of the theorem. To obtain the
third, we note that $(Y_n)_{n\in\bN}$ with $Y_n:=X_{-n}$ for all $n\in\bN$ is a Markov chain with backwards transition
probabilities as given in~\eqref{eq:bw}. The general convergence result (R2) mentioned at the end of Section~\ref{sec:summ} 
now implies that $Y_n\to z$ as $n\to\infty$, $\bP^z$-almost surely and with respect to the Doob-Martin convergence.
Hence another invocation of Theorem~\ref{thm:part} completes the proof.
\end{proof}

We return to the election algorithm: If we start with a group of size $j_k$ then the number of rounds needed
is the maximum $M_{j_k}$ of a sample of size $j_k$ from the geometric distribution with parameter $\theta$ 
if the maximum is unique, and $M_{j_k}+1$ otherwise. 
It is well known that, if $j_k\to\infty$ with $k\to\infty$, we then need
to subtract suitable values $n_k\in\bN$, with $n_k\to\infty$ as $k\to\infty$,
to obtain a tight sequence of distributions $\cL(M_{j_k}-n_k)$, $k\in\bN$, 
and that convergence in distribution only holds along specific subsequences. Indeed, in the limit a logarithmic periodicity can be 
observed; see Remark~\ref{rem:asconv} for the non-existence of strong limit results for the maxima.

On first sight, the results of this section do not seem to contribute to our understanding of this periodicity 
phenomenon that many authors found intriguing and that may have contributed to the popularity of
the election algorithm in the mathematical literature. In order to obtain a connection we introduce the 
shift operator on the new path space,
\begin{equation*}
  T:\bN_0^{\bZ}\to \bN_0^{\bZ}, \quad (i_n)_{n\in\bZ}\mapsto (i_{n+1})_{n\in \bZ}.
\end{equation*}
Here is the basic observation: Replacing $k$ by $k+1$ in~\eqref{eq:DMpart} corresponds to replacing
$z$ by $z-1$, which implies $(\bP_z)^T=\bP_{z+1}$.  
On the process side, the shift $T$ corresponds to the transition from $(X_n)_{n\in\bZ}$ to $(X_{n+1})_{n\in\bZ}$.
In particular, the probability of events such as the time from some initial number of participants to the first entry
into $\{0,1\}$ is invariant under $T$, which shows that the distribution of the duration of the algorithm depends
on $z$ only via its fractional part. 

As in the previous section we close this section with some pointers to related results.

\begin{remark}\label{rem:GW} 
(a) A sample of size $n$ from the geometric distribution with parameter $\theta$ can be obtained
by putting balls into an infinite sequence of urns in the following way: In
the first round, each of the initially $n$ balls is put into the first urn with probability $\theta$, then each of the 
remaining balls is put into the second urn with probability $\theta$, and so on. Gnedin~\cite{Gne04} introduced
a randomized version, the Bernoulli sieve, where the fixed values $\theta$ for the successive rounds are replaced 
by the values of a sequence of independent and identically distributed random variables. Interestingly, such an
additional randomisation may eliminate the oscillatory effects inherent to the classical geometric leader election;
see~\cite{GINR09} and~\cite{GnIkMa10}. 

(b) In Section~\ref{sec:summ} we dealt with what is known as the \emph{exit} boundary. The step from
Theorem~\ref{thm:part} to Theorem~\ref{thm:minusinf} corresponds to the somewhat dual notion of an
\emph{entrance} boundary; see~\cite{Doob59} and~\cite[]{KSK}. 

(c) Markov chain boundaries have been thoroughly investigated in the context of random walks on
discrete structures, a standard reference being~\cite{Woess1}. Branching (or Galton-Watson) processes
are another class of discrete time Markov chains where a lot is known about the boundaries; 
see~\cite[Chapter~II]{AthrNey}. This class is of special relevance to the above process of participant
numbers, as this process is in fact a branching process with a specific offspring distribution: There is either one 
descendant or none at all, with respective probabilities $1-\theta$ and $\theta$. For such subcritical
cases a general result on the entrance boundary has been obtained in~\cite{AlsRoe06}; in particular,
for the process itself (not the space-time version) the boundary is the torus under quite general conditions. 
\brend 
\end{remark}

{\bf Acknowledgements}. This paper was written on the occasion of the conference ftb2015  taking place at 
Universit{\'e} Libre de Bruxelles, September 9-11, 2015, 
celebrating F.\ Thomas Bruss and his contributions to applied probability.

{\small
\providecommand{\bysame}{\leavevmode\hbox to3em{\hrulefill}\thinspace}
\providecommand{\MR}{\relax\ifhmode\unskip\space\fi MR }
\providecommand{\MRhref}[2]{%
  \href{http://www.ams.org/mathscinet-getitem?mr=#1}{#2}
}
\providecommand{\href}[2]{#2}

}

\end{document}